\documentclass{amsart}

\usepackage{times, amsmath,amssymb,amsthm}
\usepackage{graphicx}
\usepackage{a4wide}
\usepackage[all]{xy}

\usepackage{color}
\definecolor{Chocolat}{rgb}{0.36, 0.2, 0.09}
\definecolor{BleuTresFonce}{rgb}{0.215, 0.215, 0.36}

\usepackage[colorlinks,final,hyperindex]{hyperref}
\hypersetup{citecolor=BleuTresFonce, linkcolor=Chocolat}

\DeclareMathAlphabet{\mathbbold}{U}{bbold}{m}{n}

\DeclareSymbolFont{rsfscript}{OMS}{rsfs}{m}{n}
\DeclareSymbolFontAlphabet{\mathrsfs}{rsfscript}

\DeclareFontFamily{OMS}{rsfs}{\skewchar\font'177}
\DeclareFontShape{OMS}{rsfs}{m}{n}{%
      <5> rsfs5
      <6> <7> rsfs7
      <8> <9> <10> rsfs10
      <10.95> <12> <14.4> <17.28> <20.74> <24.88> rsfs10
      }{}
\def\KK{\mathbb{K}}
\newcommand{\ac}{\scriptstyle \text{\rm !`}}

\DeclareMathOperator{\id}{id}
\DeclareMathOperator{\End}{End}
\DeclareMathOperator{\Hom}{Hom}

\DeclareMathOperator{\BV}{BV}

\makeatletter

\theoremstyle{plain}

\newtheorem {theorem}[subsection]{Theorem}
\newtheorem {lemma}[subsection]{Lemma}
\newtheorem {corollary}[subsection]{Corollary}

\newtheorem {proposition}[subsection]{Proposition}
\newtheorem*{theoremintro}{Theorem}

\theoremstyle{definition}

\newtheorem {definition}[subsection]{Definition}

\newtheorem{remarks}[subsection]{Remarks}



\thanks{S.S. was supported by the Netherlands Organisation for Scientific Research. B.V. was supported by the ANR HOGT grant.}

\begin{document}

\title[De Rham cohomology and Frobenius manifolds]{De Rham cohomology and homotopy Frobenius manifolds} 

\author{Vladimir Dotsenko}
\address{Mathematics Research Unit, 
	 University of Luxembourg, 
	 Campus Kirchberg, 
	 6, Rue Richard Couden\-hove-Kalergi, 
	 L-1359 Luxembourg, 
	 Grand Duchy of Luxembourg}
\email{vladimir.dotsenko@uni.lu}

\author{Sergey Shadrin}
\address{Korteweg-de Vries Institute for Mathematics, University of Amsterdam, P. O. Box 94248, 1090 GE Amsterdam, The Netherlands}
\email{s.shadrin@uva.nl}

\author{Bruno Vallette}
\address{Laboratoire J.A.Dieudonn\'e, Universit\'e de Nice Sophia-Antipolis, Parc Valrose, 06108 Nice Cedex 02, France}
\email{brunov@unice.fr}

\begin{abstract}
We endow the de Rham cohomology of any Poisson or Jacobi manifold with a natural homotopy Frobenius manifold structure. This result relies on a minimal model theorem for multicomplexes and a new kind of a Hodge degeneration condition.
\end{abstract}

\maketitle

\setcounter{tocdepth}{1}

\tableofcontents

\section*{Introduction}\label{sec:intro}

Jean-Louis Koszul defined in \cite{Koszul85}, for the first time, the general notion of a commutative algebra equipped with a square-zero differential operator of order $2$. This algebraic structure is 
now called a Batalin--Vilkovisky algebra. It is straightforward to extend this definition to the differential graded framework by requiring an extra compatible differential. One of the main example given by Koszul is the de Rham cochain complex of a Poisson manifold. 

A Frobenius manifold \cite{Manin99} is an algebraic structure that amounts to the operadic action of the homology of the Deligne--Mumford--Knudsen compactification of the moduli space of genus $0$ curves $H_\bullet(\overline{\mathcal{M}}_{0, n+1})$. 
Motivated by ideas from string theory~\cite{BCOV94}, Barannikov and Kontsevich showed in~\cite{BarannikovKontsevich98} that the Dolbeault cohomology of a Calabi--Yau manifold carries a natural Frobenius manifold structure; this demonstrated a crucial role Frobenius manifolds play in the formulation of one of the versions of the Mirror Symmetry conjecture~\cite{CaoZhou01}. 

Using the methods of Barannikov and Kontsevich together with a result of Mathieu \cite{Mathieu95}, Merkulov \cite{Merkulov98} endowed the de Rham cohomology of a symplectic manifold, satisfying the  hard Lefschetz condition, with a natural Frobenius manifold structure.

\smallskip

Getzler proved in \cite{Getzler95} that the Koszul dual of the operad $H_\bullet(\overline{\mathcal{M}}_{0, n+1})$ is the cohomology of the moduli space of genus $0$ curves $H^{\bullet+1}({\mathcal{M}}_{0, n+1})$. A coherent action of the latter spaces defines the notion of homotopy Frobenius manifold, with the required homotopy properties. Notice that a homotopy Frobenius manifold structure on a graded vector space, i.e. a chain complex with trivial differential, is made up of  an infinite sequence of strata of multilinear operations, whose first stratum forms a Frobenius manifold. 

The purpose of this paper is to prove the following theorem.

\begin{theoremintro}[\ref{PoissonFrob}, \ref{JacobiFrob}]
The de Rham cohomology of a Poisson manifold (respectively a  Jacobi manifold) carries a natural homotopy Frobenius manifold structure, which extends the product induced by the wedge product.
\end{theoremintro}

This theorem extends Merkulov's result in three directions. First, it holds for any Poisson manifolds. Then, it provides us with higher geometrical invariants which faithfully encodes the initial algebraic structure since it allows us to reconstruct the homotopy type of the initial Batalin--Vilkovisky algebra. Finally, it also extends from Poisson manifolds to Jacobi manifolds.  Note that in the case of Jacobi manifolds, the setup of dg $\BV$-algebras is not sufficient anymore, and one has to use commutative homotopy $\BV$-algebras~\cite{Kravchenko00} instead. 

\smallskip

Furthermore, Cao and Zhou found a natural Frobenius manifold structure on the Dolbeault cohomology of a closed K\"ahler manifold~\cite{CaoZhou00}. They proved that, for a compact K\"ahler manifold, this Frobenius manifold structure is isomorphic to that on the de Rham cohomology~\cite{CaoZhouId99}. They also found a natural Frobenius manifold structure on equivariant cohomology of closed K\"ahler manifolds~\cite{CaoZhouEq99}. Finally, in \cite{CaoZhouQuant99}, they defined quantum de Rham cohomology of Poisson manifolds and its Laurent series version (the latter one closely related to the cyclic homology of Poisson manifolds~\cite{Pap00}). Then they constructed natural Frobenius manifold structures on the quantum de Rham cohomology of closed K\"ahler manifolds and Laurent quantum cohomology of compact symplectic manifolds. 

The method of the present paper can be applied \emph{mutatis mutandis} to obtain appropriate generalisations (with shorter proofs) of the abovementioned results of Cao and Zhou as well.

\smallskip

To prove our main result, we  develop further the homotopy theory of multicomplexes~\cite{Lapin01,Meyer78}. The notion of a multicomplex 
is a certain lift of the notion of a spectral sequence. We prove a minimal model theorem for multicomplexes, which amounts to a decomposition into a product of a minimal one and an acyclic trivial one. Furthermore, we introduce a new condition, called gauge Hodge condition, which ensures the uniform vanishing of the induced BV-operator (and its higher homotopies) on the underlying homotopy groups.  
This gauge Hodge condition, applied to the classical case of a bicomplex spectral sequence, gives a necessary and sufficient condition for that spectral sequence to degenerate at the first page.

The gauge condition is naturally suggested by the Givental action formalism we used to work with commutative homotopy $\BV$-algebras in~\cite{DotsenkoShadrinVallette11}. The idea of using gauge-type arguments to prove homotopical results is not completely new.
In particular, the operator $\Delta=Jd_{DR}J$ in complex geometry \cite{DGMS75} and generalised complex geometry \cite{Cavalcanti05,Cavalcanti06}, once written as $-Jd_{DR}J^{-1}$, can be viewed as gauge equivalent to $-d_{DR}$. Formulas ensuring the degeneration of appropriate spectral sequences for cyclic homology of Poisson manifolds \cite{Pap00} and quantum de Rham cohomology of Poisson manifolds \cite{Shur04} have a gauge symmetry flavour to them as well. Finally, the notion of gauge equivalence for Frobenius manifolds is studied in detail in Cao and Zhou~\cite{CaoZhou2003}, where it is used to prove that the construction of Barannikov and Kontsevich applied to two quasi-isomorphic dg $\BV$-algebras yields two Frobenius manifold structures that can be identified with one another.

\subsection*{Layout.} The paper is organised as follows. The first section deals with the  homotopy properties of mixed complexes and multicomplexes. We recall the homotopy transfer theorem for multicomplexes and we prove a minimal model theorem. In Section~\ref{sec:StrongHodge}, we introduce the gauge Hodge condition, and prove that its fulfilment is equivalent to the existence of a Hodge-to-de Rham degeneration data. In Section~\ref{sec:Poisson}, we construct a natural homotopy Frobenius manifold structure on the de Rham cohomology of any Poisson manifold. In Section~\ref{sec:Jacobi}, we do the same for basic de Rham cohomology of Jacobi manifolds, where the proof is very similar to the Poisson case, and for the whole de Rham cohomology of Jacobi manifolds, where the setup is more subtle, and commutative homotopy $\BV$-algebras enter the story.   

\subsection*{Conventions.} Thoughout the text, we work over a field $\KK$ of characteristic $0$.

\subsection*{Acknowledgements.} The second and the third author would like to thank the University of Luxembourg for the excellent working conditions enjoyed during their visits there.

\section{Homotopy theory of multicomplexes}\label{sec:MixedMulti}

\begin{definition}[Mixed complex and multicomplex]
A \emph{mixed complex} $(A, d, \Delta$) is a graded vector space $A$ equipped with two linear operators $d$ and $\Delta$ of respective degree $-1$ and $1$, satisfying
$$d^2=\Delta^2=d\Delta+\Delta d=0 \ .$$
A \emph{multicomplex} $(A, d=\Delta_0, \Delta_1, \Delta_2, \ldots )$ is a graded vector space $A$ endowed with a family of linear operators of respective degree $|\Delta_n|=2n-1$ satisfying 
$$\sum_{i=0}^n \Delta_i \Delta_{n-i} =0, \quad \text{for}\ \ n\ge 0\ . $$
\end{definition}

Since   $d=\Delta_0$ squares to zero, $(A, d)$ is a chain complex. We call \emph{homotopy groups} of  a multicomplex $A$, the underlying homology groups $H(A,d)$. A mixed complex is a multicomplex where all the higher operators $\Delta_n=0$ vanish, for $n\ge 2$. The notion of multicomplex is the notion of mixed complex \emph{up to homotopy} according to the Koszul duality theory, see \cite[Section~$10.3.17$]{LodayVallette12}.

\begin{definition}[$\infty$-morphism]
An \emph{$\infty$-morphism} $f\colon A \rightsquigarrow A'$ of multicomplexes is a family of linear maps $\{f_n\colon  A \to A' \}_{n \ge 0}$ of respective degree $|f_n|=2n$ satisfying 
$$\sum_{k+l=n}   f_k \Delta_l \, =  \sum_{k+l=n}   \Delta'_k f_l\, , \quad \text{for}\ \  n\ge 0\ . $$
The composite of two $\infty$-morphisms $f\colon A \rightsquigarrow A'$ and $g\colon A' \rightsquigarrow A''$ is given by 
$$(gf)_n := \sum_{k+l=n} g_k f_l\, , \quad \text{for}\ \  n\ge 0 \ . $$
The associated category is denoted by $\infty$-\textsf{multicomp}.
\end{definition}
Notice that $f_0\colon (A, d) \to (A', d')$ is a chain map. When the first map $f_0$ is a quasi-isomorphism (respectively an isomorphism), the $\infty$-morphism $f$ is called an \emph{$\infty$-quasi-isomorphism} (respectively an \emph{$\infty$-isomorphism}), and denoted $A \stackrel{\sim}{\rightsquigarrow} A'$ (respectively $A \stackrel{\cong}{\rightsquigarrow}\allowbreak A'$). The invertible morphisms of the category $\infty$-\textsf{multicomp} are the $\infty$-isomorphisms. An $\infty$-isomorphism whose first component is the identity map is called an \emph{$\infty$-isotopy} and denoted $A \stackrel{=}{\rightsquigarrow} A'$. \\

A \emph{homotopy retract} consists of the following data 
\begin{eqnarray*}
&\xymatrix{     *{ \quad \ \  \quad (A, d_A)\ } \ar@(dl,ul)[]^{h}\ \ar@<0.5ex>[r]^{p} & *{\
(H,d_H)\quad \ \  \ \quad }  \ar@<0.5ex>[l]^{i}\ , }&
\end{eqnarray*}
where $p$ is a  chain map, where $i$ is a quasi-isomorphism, and where $h$ has degree $1$, satisfying 
$$i p- \textrm{id}_A =d_A  h+ h  d_A \ . $$
 If moreover $pi=\textrm{id}_H$, then it is called a \emph{deformation retract}. 

\begin{proposition}[Homotopy Transfer Theorem \cite{Lapin01}]\label{prop:HTT}
Given a homotopy retract data between two chain complexes $A$ and $H$, and a multicomplex structure 
$\{ \Delta_n \}_{n \ge 1}$ on $A$, the following formulae define a multicomplex structure on $H$
\begin{equation}\label{eq:HTTformulae}
 \Delta'_n:=  \sum_{i_1+\dots+i_k=n }    
 p \Delta_{i_1} h \Delta_{i_2} h \ldots h \Delta_{i_k} i , \quad \text{for} \ \ n\ge 1
\ , 
\end{equation} 
an $\infty$-quasi-isomorphism $i_\infty\colon H \stackrel{\sim}{\rightsquigarrow}A$, which extends the map $i$
$$i_n:=  \sum_{i_1+\dots+i_k=n }    
 h \Delta_{i_1} h \Delta_{i_2} h \ldots h \Delta_{i_k} i , \quad \text{for} \ \ n\ge 1
\ , $$ 
 and an $\infty$-quasi-isomorphism $p_\infty\colon A \stackrel{\sim}{\rightsquigarrow} H $,  which extends the map $p$
$$p_n:=  \sum_{i_1+\dots+i_k=n }    
 p \Delta_{i_1} h \Delta_{i_2} h \ldots h \Delta_{i_k} h , \quad \text{for} \ \ n\ge 1
\ . $$
\end{proposition}

\begin{proof} The proof is a straightforward computation. One can also prove it using the following interpretation. 
Let $D:=T(\Delta)/(\Delta^2)$ be the algebra of dual numbers generated by one element of degree $1$. So a $D$-module is a mixed complex. The Koszul dual coalgebra $D^{\ac}=T^c(\delta)$ is the free coalgebra on a degree $2$ element $\delta:=s\Delta$, where $s$ stands for the homological suspension. The cobar construction of $D^{\ac}$ is equal to 
 $$
D_\infty:=\Omega D^{\ac}=T(s^{-1} (\bigoplus_{n\ge 1} \KK \,  \delta^n)).
 $$ 
So a $D_\infty$-module is a multicomplex. Using this interpretation, the proposition is a direct consequence of the general Homotopy Transfer Theorem of \cite[Section~$10.3$]{LodayVallette12}.
\end{proof}

\begin{definition}[Hodge-to-de Rham degeneration]
Let  $(A, d, \Delta_1, \Delta_2, \ldots )$  be a multicomplex. A {\em  Hodge-to-de Rham degeneration data}  consists of a homotopy retract 
\begin{eqnarray*}
\xymatrix{     *{ \quad \ \  \quad (A, d)\ } \ar@(dl,ul)[]^{h}\ \ar@<0.5ex>[r]^-{p} & *{\
(H(A),0)\ ,\quad \ \  \ \quad }  \ar@<0.5ex>[l]^-{i} } & \text{satisfying} \\
 \displaystyle \sum_{i_1+\dots+i_k=n }    
 p \Delta_{i_1} h \Delta_{i_2} h \ldots h \Delta_{i_k} i=0 , & \text{for}\  n\ge 1. 
\end{eqnarray*}
\end{definition}

This data amounts to the vanishing of all the transferred operators  $\Delta'_n$ on the underlying homotopy groups of a multicomplex. \\

To any multicomplex $(A, \Delta_0, \Delta_1, \Delta_2, \ldots )$, one associates the following chain complex.  
Let $C_{p,q}:=\allowbreak A_{p-q}$ and $\partial_r:=\Delta_r : C_{p,q}\to C_{p-1+r,q-r}$. We consider the total complex $\widehat{\mathrm{Tot}}(C)_n:=\prod_{p+q=n} C_{p,q}$, equipped with the differential $\partial:=\sum_{r\ge 0} \partial_r$. (The degrees of the respective $\Delta_n$ ensures that $\partial$ has degree $-1$.)
 The row filtration $F_n$ defined by considering the 
$C_{\bullet, k}$, for $k \leq -n$, provides us with a decreasing filtration of the total complex and thus with a spectral sequence $E^r(A)$. 

\begin{proposition}[Degeneration at page $1$]\label{prop:Degeneration}
The spectral sequence $E^r(A)$ associated to a multicomplex $(A, d=\Delta_0, \Delta_1, \Delta_2, \ldots )$ degenerates at the first page if and only if there exists a Hodge-to-de Rham degeneration data.
\end{proposition}

\begin{proof}
If the differentials $d^r$ vanish for $r\ge 1$, then $E^1=E^2=\ldots=H(A, d)$. In this case, the formulae \cite[Chapter~$\text{III}$]{BottTu82} for the $d^r$ are equal to the formulae defining  the transferred $\Delta'_r$. The other way round, one sees by induction from $r=1$ that $E^r=H(A,d)$ and that $d^r=\Delta'_r$. 
\end{proof}

In the case of a mixed complex, $C_{\bullet, \bullet}$ is a bicomplex. So the Hodge-to-de Rham condition 
is equivalent to degeneration of  the usual bicomplex spectral sequence at the first page. This is the case for the classical Hodge-to-de Rham spectral sequence of compact K\"ahler manifolds.\\

A multicomplex $(A, d=\Delta_0, \Delta_1, \Delta_2, \ldots )$ is called \emph{minimal} when $d=\Delta_0=0$. It is called \emph{acyclic} when the underlying chain complex $(A, d)$ is acyclic, and it is called \emph{trivial} when $\Delta_n=0$, for $n\ge 1$.

\begin{theorem}[Minimal model]\label{thm:MinimalModel}
In the category $\infty$\textsf{-multicomp}, any multicomplex $A$ is $\infty$-isomorphic to the product of a minimal multicomplex $H=H(A)$, given by the transferred structure, with an acyclic trivial multicomplex $K$.
\end{theorem}

\begin{proof}
This theorem is a direct consequence of \cite[Theorem~$10.4.5$]{LodayVallette12} applied to the Koszul algebra $D$.
More precisely, we consider a choice of representatives for the homology classes $H(A)\cong H \subset A$ and a complement $K \subset A$ of it. This decomposes the chain complex $A = H\oplus K$, where the differential on $H=H(A)$ is trivial and where the chain complex $K$ is acyclic.  Let us denote the respective projections by $p\colon A \twoheadrightarrow H$ and by $q\colon A \twoheadrightarrow K$. This induces the following homotopy retract 
\begin{eqnarray*}
&\xymatrix{     *{ \quad \ \  \quad (A, d_A)\ } \ar@(dl,ul)[]^{h}\ \ar@<0.5ex>[r]^{p} & *{\
(H,0) \ .\quad \ \  \ \quad }  \ar@<0.5ex>[l]^{i}\ , }&
\end{eqnarray*}
Using Formula~(\ref{eq:HTTformulae}) of Proposition~\ref{prop:HTT}, we endow $H$ with the transferred multicomplex structure.  So $(H, 0 , \{\Delta'_n\}_{n\ge 1})$ is a minimal multicomplex and $(K, d_K, 0)$ is an acyclic trivial multicomplex. Their product in the category $\infty$\textsf{-multicomp} is given by 
$(H\oplus K, d_K , \allowbreak \{\Delta'_n\}_{n\ge 1})$. The projection $q$ extends to an $\infty$-morphism $q_\infty$ by 
$q_n := q h \Delta_n  $, for $n\ge 1$. By the categorical property of the product, the maps $p_\infty$ and $q_\infty$ induce the following $\infty$-isomorphism $r\colon A {\rightsquigarrow} H \oplus K  $, explicitly given by $r_0:=p+q$ and by 
\begin{eqnarray}  \label{eqn:InftyMorphiR}
r_n:=p_n+q_n= \sum_{i_1+\dots+i_k=n } p \Delta_{i_1} h \Delta_{i_2} h \ldots h \Delta_{i_k} h + qh\Delta_n , \quad \text{for} \ \ n\ge 1 \ . 
\end{eqnarray}
\end{proof}

\section{Gauge Hodge condition}\label{sec:StrongHodge}
We consider the algebra $\End(A)[[z]]:=\Hom(A,A)\otimes \KK[[z]]$
of formal power series with coefficients in the endomorphism algebra of $A$. One can view the $\infty$-endomorphisms of a mulitcomplex $A$ as elements of $\End(A)[[z]]$. Under this interpretation, their composite corresponds to the product of the associated series. 

\begin{theorem}\label{thm:main}
A multicomplex  $(A, d, \Delta_1, \Delta_2, \ldots )$  admits a Hodge-to-de Rham degeneration data if and only if there exists an
element $R(z):=\sum_{n\ge1}R_n z^n$ in $\End(A)[[z]]$  satisfying 
\begin{eqnarray}\label{eqn:Conj}
e^{R(z)} \, d\,  e^{-R(z)}= d +  \Delta_1  z + \Delta_2 z^2+\cdots \ .
\end{eqnarray}
\end{theorem}

\begin{proof} The proof is built from  the following three  equivalences. 

\begin{itemize}

\item[\underline{Step 1.}] We first prove that there exists a series $R(z):=\sum_{n\ge1}R_n z^n \in \End(A)[[z]]$ satisfying 
Condition~(\ref{eqn:Conj}) if and only if there exists an $\infty$-isotopy 
$$(A, d, 0, \ldots) \stackrel{=}{\rightsquigarrow} \allowbreak (A, d, \Delta_1, \Delta_2, \ldots )$$ between $A$ with trivial structure and $A$ with its multicomplex structure. 

\noindent Condition~(\ref{eqn:Conj}) is equivalent to 
$e^{R(z)} \, d= (d +  \Delta_1  z + \Delta_2 z^2+\cdots ) e^{R(z)}$, which means that $e^{R(z)}$ is the required $\infty$-isotopy. 

\item[\underline{Step 2.}] Given a deformation retract for $A$ onto its underlying homotopy groups $H(A)$, there exists an $\infty$-isotopy $(A, d, 0, \ldots) \stackrel{=}{\rightsquigarrow} (A, d, \Delta_1, \Delta_2, \ldots )$  if and only if 
there exists an $\infty$-isotopy 
 $$
(H(A), 0, 0, \ldots)\allowbreak \stackrel{=}{\rightsquigarrow} (H(A), 0, \Delta'_1, \Delta'_2, \ldots )\ .
 $$ 

\noindent
The homotopy transfer theorem of Proposition~\ref{prop:HTT} provides us with the following diagram in $\infty$\textsf{-multicomp}.
$$\xymatrix@R=35pt@C=35pt{(A, d, 0, \ldots)  \ar@{~>}[r]^(0.45)=_(0.45)\varphi  &      (A, d, \Delta_1, \Delta_2, \ldots )    \ar@{~>}[d]^\sim_{p_\infty} \\
(H(A), 0, 0, \ldots)   \ar@{~>}[r]^(0.45)=_(0.45)\psi  \ar@{~>}[u]^i_\sim    &    (H(A), 0, \Delta'_1, \Delta'_2, \ldots )  
} $$
So given an $\infty$-isotopy $\varphi$, the composite $\psi:=p_\infty \, \varphi \, i $ is an $\infty$-isotopy. In the other way round, we suppose given an $\infty$-isotopy $\psi$. The map $p+q\colon A \to H(A)\oplus K$ is a map of chain complexes, hence it is an $\infty$-isomorphism between these two trivial multicomplexes. Then, the map $\psi + \id_K$ defined by $\id_H + \id_K$, for $n=0$, and by $\psi_n$, for $n\ge 1$ defines an $\infty$-isomorphism between $H(A)\oplus K$ with trivial multicomplex structure to $H(A)\oplus K$ with the transferred structure. Finally, we consider the inverse $\infty$-isomorphism $r^{-1}\colon H(A)\oplus K \stackrel{\cong}{\rightsquigarrow} A$ of the $\infty$-isomorphism $r$ given at (\ref{eqn:InftyMorphiR}) in the proof of Theorem~\ref{thm:MinimalModel}. The composite $r^{-1}\, (\psi + \id_K)\, (p+q) $ of these three maps provides us with the required $\infty$-isotopy.  

\item[\underline{Step 3.}] Let us now prove that an $\infty$-isotopy 
 $$
(H(A), 0, 0, \ldots)\allowbreak \stackrel{=}{\rightsquigarrow} (H(A), 0, \Delta'_1,\allowbreak \Delta'_2, \ldots )
 $$ 
exists if and only if the (transferred) operators $\Delta'_n=0$ vanish for $n\ge 1$. 

\noindent
Let us denote by $f\colon (H(A), 0, 0, \ldots)  \stackrel{=}{\rightsquigarrow} (H(A), 0, \Delta'_1,\Delta'_2,  \ldots)$ the given $\infty$-isotopy.  The defining condition 
$$\sum_{k+l=n}   \Delta'_k f_l =\sum_{k+l=n}   f_k \Delta_l  =0\, , \quad \text{for}\ \  n\ge 1$$
implies $\Delta'_n=0$, for $n\ge 1$, by a direct induction. In the other way, the identity $\id_{H(A)}$ provides us the required $\infty$-isotopy. 
\end{itemize}

\end{proof}

We call the \emph{gauge Hodge condition} the existence of a series $R(z)\in \End(A)[[z]]$ satisfying the conjugation condition~(\ref{eqn:Conj}).

\begin{remarks}$ \ $
\begin{itemize}
\item[$\diamond$] 
Notice that this proof actually shows that, under the gauge Hodge condition,  \emph{every}  deformation retract is a Hodge-to-de Rham degeneration data. In this case, the transferred multicomplex structure vanishes uniformly, i.e. independently of the choices of representatives of the homotopy groups. This theorem solves a question that we raised at the end of \cite{DotsenkoShadrinVallette11}.

\item[$\diamond$] When $(A,d,\Delta)$ is a mixed complex equipped with a Hodge-to-de Rham degeneration data, the seriez $R(z)$ defined by the formula 
 $$
R(z):= -\log(1 - h\Delta  z + \sum_{n\ge 1 }  i p (\Delta h)^n z^n ) \ ,    
 $$
satisfies Relation~(\ref{eqn:Conj}) of Theorem~\ref{thm:main}. Explicitly, $R(z)=\sum_{n\ge 1} r_n z^n$, where
 $$
r_n=\frac{(h\Delta)^n}{n}-n\sum_{l=1}^{n}\frac{(h\Delta)^{l-1} i p (\Delta h)^{n-l+1}}{l}\ .
 $$

\item[$\diamond$] A $\BV$-algebra equipped with a series $R(z):=\sum_{n\ge1}R_n z^n$ satisfying Relation~(\ref{eqn:Conj}) is called a \emph{$BV/\Delta$-algebra} in \cite{KhoroshkinMarkarianShadrin11}, where this notion is studied in detail. 
\end{itemize}
\end{remarks}

\section{De Rham cohomology of Poisson manifolds}\label{sec:Poisson}

\begin{definition}[Frobenius manifold, \cite{Manin99}]
A \emph{(formal)  Frobenius manifold} is an algebra over the operad 
$ H_\bullet(\overline{\mathcal{M}}_{0, n+1})$ made up of the homology of the  Deligne-Mumford-Knudsen moduli spaces of stable genus $0$ curves.  
\end{definition}
This algebraic structure amounts to giving a collection of symmetric multilinear maps $\mu_n\colon A^{\otimes n} \to A$, for $n\ge 2$, of degree $|\mu_n|:=2(n-2)$ satisfying some quadratic relations, see \cite{Manin99} for instance. It is also called an \emph{hypercommutative algebra} in the literature. (Notice that we do not require here any 
non-degenerate pairing nor any unit). 

The operad $H_\bullet(\overline{\mathcal{M}}_{0, n+1})$ is Koszul, with Koszul dual cooperad $H^{\bullet+1}({\mathcal{M}}_{0, n+1})$, the cohomology groups  of the moduli spaces of genus $0$ curves. Algebras over the linear dual operad $H_{\bullet}({\mathcal{M}}_{0, n+1})$ are called \emph{gravity algebras} in the literature. 
The operadic cobar construction  $\Omega H^{\bullet}({\mathcal{M}}_{0, n+1})\stackrel{\sim}{\to} H_\bullet(\overline{\mathcal{M}}_{0, n+1})$ provides a resolution of the former operad, see \cite{Getzler95}.

\begin{definition}[Homotopy Frobenius manifold]
A  \emph{homotopy Frobenius manifold} is an algebra over the operad $\Omega H^{\bullet}({\mathcal{M}}_{0, n+1})$. 
\end{definition}

The operations defining such a structure are parametrised by $H^{\bullet}({\mathcal{M}}_{0, n+1})$. 
 Hence, a homotopy Frobenius manifold structure on a chain complex with trivial differential is made up of an infinite sequence of strata of multilinear operations, whose first stratum forms a Frobenius manifold. 
 
\begin{definition}[dg $\BV$-algebra]
A \emph{dg $\BV$-algebra} $(A, d,  \wedge, \Delta)$ is a differential graded commutative algebra equipped with a square-zero degree $1$ operator $\Delta$ of order less than $2$.
\end{definition}

The data of a dg $\BV$-algebra amounts to a mixed complex data $(A,d, \Delta)$ together with a compatible commutative product. We refer the reader to \cite[Section~$13.7$]{LodayVallette12} for more details on this notion. 

\smallskip

To any homotopy Frobenius manifold $H$, we can associate a \emph{rectified} dg $\BV$-algebra $Rec(H)$, see \cite[Section~$6.3$]{DrummondColeVallette11}. 

\begin{theorem}[\cite{DrummondColeVallette11}]\label{thm:HomotopyFrob}
Let $(A, d,  \wedge, \Delta)$ be a dg $\BV$-algebra equipped with a Hodge-to-de Rham degeneration data. 
The underlying homotopy groups $H(A,d)$ carry a homotopy Frobenius manifold structure, whose rectified dg $\BV$-algebra is homotopy equivalent to $A$.
\end{theorem}

This result shows that the transferred homotopy Frobenius manifold faithfully encodes the homotopy type of the dg $\BV$-algebra $A$. It provides a refinement of a result of Barannikov and Kontsevich \cite{BarannikovKontsevich98}, where only the underlying Frobenius manifold structure is considered. This first stratum of operations can be described in terms of sums of labelled graphs, see \cite{LosevShadrin07}.

\begin{proposition}[\cite{Koszul85}]\label{prop:PoissonKoszul}
Let $(M, \omega)$ be a Poisson manifold. Its de Rham complex $(\Omega^\bullet(M), d_{DR}, \wedge, \Delta)$ is a dg $\BV$-algebra, with the operator $\Delta$ defined by 
 $$
\Delta:=i(\omega)d_{DR}-d_{DR}i(\omega)=[i(\omega),d_{DR}], 
 $$
where $i(-)\colon\Omega^\bullet(M)\to\Omega^{\bullet-2}(M)$ denotes the contraction operator.
\end{proposition}

In particular, $\Omega^\bullet(M)$ becomes a mixed complex, the \emph{canonical double complex} of Brylinski~\cite{Brylinski88}.

\smallskip

Koszul's proof of this result relies on the following relation between the contraction operators, the Schouten--Nijenhuis bracket, and the de Rham differential, which we shall use throughout the paper.

\begin{proposition}[\cite{Koszul85,Marle97}]
For every smooth manifold $M$, and every polyvector fields $\omega_1, \omega_2$,
 $$
i([\omega_1,\omega_2])=-[[i(\omega_2),d],i(\omega_1)].
 $$ 
\end{proposition}

\begin{theorem}\label{PoissonFrob}
The de Rham cohomology of a Poisson manifold~$(M, \omega)$ carries a natural homotopy Frobenius manifold structure, whose rectified dg $\BV$-algebra is homotopy equivalent to $(\Omega^\bullet(M), d_{DR}, \wedge, \Delta)$.
\end{theorem}

\begin{proof}
The operators of Proposition~\ref{prop:PoissonKoszul} satisfy
 $$[i(\omega),[i(\omega),d_{DR}]]= -[[i(\omega), d_{DR}],   i(\omega)]= i([\omega, \omega])=0 \ , $$
 where $i(-)$ denotes the contraction of differential forms by vector fields.
This, together with the fact that 
 $$
e^{R(z)} \, d\,  e^{-R(z)}=e^{\mathop{\mathrm{ad}}_{R(z)}}(d), \ \text{for any}\  R(z)\in \End(A)[[z]]
 $$
immediately implies that
 $$ e^{i(\omega)z} \, d_{DR}\,  e^{-i(\omega)z}= d_{DR} +  \Delta  z. $$
So by Theorem~\ref{thm:main}, $\Omega^\bullet(M)$ admits a Hodge-to-de Rham degeneration data, and Theorem~\ref{thm:HomotopyFrob} applies.
\end{proof}

This result refines the Lie formality theorem of \cite{SharyginTalalev08}, since the transferred $L_\infty$-algebra structure is trivial. 
Note that gauge theoretic methods were already used (independently) in 
 \cite{FiorenzaManetti11}
  to obtain a conceptual proof of the former theorem.

\begin{corollary}[\cite{FernandezIbanezLeon}]
For every Poisson manifold~$M$, the spectral sequence for the double complex $(\Omega^\bullet(M), d_{DR}, \Delta)$ degenerates on the first page.
\end{corollary}

\begin{proof}
By Proposition~\ref{prop:Degeneration}.
\end{proof}

Using the same argument as in Theorem~\ref{PoissonFrob} with $\Delta$ replaced by~$h\Delta$, one can prove the following result, which generalises and simplifies the proofs of the respective results of~\cite{CaoZhouQuant99,Shur04}. We refer to~\cite{CaoZhouQuant99} for respective definitions. The only warning we wish to make here is that in this case one must work over the commutative ring $\KK[h]$ instead of working over a field. However, the differential $\Delta_0=d_{DR}$ on the quantum de Rham complex $\Omega^\bullet(M)[h]$ does not depend on~$h$, which guarantees the projectivity of all modules that have to be assumed projective in order for the homotopy transfer machinery to work.

\begin{theorem}\label{QuantumPoissonFrob}
The quantum de Rham cohomology $Q_hH^*_{DR}(M)$ of a Poisson manifold~$(M, \omega)$ is a deformation quantisation of its de Rham cohomology:
 $$
Q_hH^*_{DR}(M)\cong H^*_{DR}(M)[h].
 $$
It carries a natural homotopy Frobenius manifold structure, whose rectified dg $\BV$-algebra is homotopy equivalent to $(\Omega^\bullet(M)[h], d_{DR}, \wedge, h\Delta)$.
\end{theorem}

\section{De Rham cohomology of Jacobi manifolds}\label{sec:Jacobi}

It turns out that the above argument can be generalised even further, namely to Jacobi manifolds. 

\begin{definition}[Jacobi manifold \cite{Lichnerowicz78}]
A \emph{Jacobi manifold} is a smooth manifold $M$ equipped with a pair  
 $$
(\omega,E)\in\Gamma(\Lambda^2(TM))\times\Gamma(TM),
 $$ 
for which
 $$
[\omega,\omega]=2E\wedge\omega, \quad [E,\omega]=0. 
 $$
\end{definition}

We consider again the space of differential forms equipped with the order $2$ operator $\Delta:=[i(\omega), d_{DR}]$. It is easy to check that it anticommutes with the de Rham differential: $d_{DR}\Delta+\Delta d_{DR}=0$. Unlike the previous case of Poisson manifolds, the operator $\Delta$ does not square to $0$ on every form of a Jacobi manifold. 

An obvious way to generalise Theorem~\ref{PoissonFrob} is hinted at by a result on a spectral sequence degeneration from~\cite{ChineaMarreroLeon}, and is concerned with basic differential forms.

\begin{definition}[Basic differential form \cite{ChineaMarreroLeon}]
A differential form $\alpha\in\Omega^{k}(M)$ is said to be \emph{basic} if 
 $$
i(E)(\alpha)=i(E)(d_{DR}\alpha)=0 \ .
 $$
We denote the space of basic differential forms by $\Omega_B^\bullet(M)$. Clearly, $(\Omega_B^\bullet(M), d_{DR},  \wedge)$ is a dg sub-algebra of $(\Omega^\bullet(M), d_{DR},  \wedge)$.
\end{definition}

The following is proved in~\cite{ChineaMarreroLeon}; we reproduce the proof here since the computations will be used in our further result, and some of the formulas in the proof are different due to a mismatch of sign conventions~\cite{Marle97}.

\begin{lemma}[\cite{ChineaMarreroLeon}]\label{DeltaSquaredOnJacobi}
The operator $\Delta$ preserves the space of basic differential forms and its  restriction to it squares to zero
$$\Delta^2|_{\Omega_{B}^\bullet(M)}=0 \ . $$
\end{lemma}

\begin{proof}
First, let us note that the de Rham differential anti-commutes with $\Delta$:
\begin{equation}\label{anticommute}
d_{DR}\Delta+\Delta d_{DR}=d_{DR}(i(\omega)d_{DR}-d_{DR}i(\omega))+(i(\omega)d_{DR}-d_{DR}i(\omega))d_{DR}=0\ .
\end{equation}
A similar computation shows that the de Rham differential anti-commutes with $i(E)$:
\begin{multline*}
i(E)\Delta+\Delta i(E)=i(E)(i(\omega)d_{DR}-d_{DR}i(\omega))+(i(\omega)d_{DR}-d_{DR}i(\omega))i(E)=\\=
i(\omega)i(E)d_{DR}+(d_{DR}i(E)-L_E)i(\omega)+i(\omega)(-i(E)d_{DR}+L_E)-d_{DR}i(E)i(\omega)=\\=
-L_E i(\omega)+i(\omega)L_E=-i([E,\omega])=0.
\end{multline*}
This implies that whenever $\alpha$ is a basic form, the form $\Delta\alpha$ is also basic, since $i(E)\Delta\alpha=-\Delta i(E)\alpha=0$ and
$i(E)d_{DR}\Delta\alpha=-i(E)\Delta d_{DR}\alpha=\Delta i(E)d_{DR}\alpha=0$, so basic forms are stable under the operator $\Delta$. To prove that $\Delta^2=0$ on basic forms, we note that
\begin{equation}\label{adjointsquared}
[i(\omega),\Delta]=-[\Delta, i(\omega)]=-[[i(\omega),d_{DR}],i(\omega)]=i([\omega,\omega])=2i(E\wedge\omega)=2i(E)i(\omega).
\end{equation}
Furthermore, we have
\begin{equation}\label{deltasquaredtemp}
\Delta^2=d_{DR}i(\omega)d_{DR}i(\omega)+i(\omega)d_{DR}i(\omega)d_{DR}-d_{DR}i(\omega)^2d_{DR}. 
\end{equation}
To simplify that latter expression, we compute
\begin{multline*}
[i(\omega),\Delta]=i(\omega)(i(\omega)d_{DR}-d_{DR}i(\omega))-(i(\omega)d_{DR}-d_{DR}i(\omega))i(\omega)=\\
=-2i(\omega)d_{DR}i(\omega)+i(\omega)^2d_{DR}+d_{DR}i(\omega)^2,
\end{multline*}
hence
\begin{equation}\label{i-omega-i}
2i(\omega)d_{DR}i(\omega)=i(\omega)^2d_{DR}+d_{DR}i(\omega)^2-2i(E)i(\omega),
\end{equation}
which allows us to simplify Formula~\eqref{deltasquaredtemp} into
\begin{equation}\label{deltasquared}
\Delta^2+i(E)i(\omega)d_{DR}+d_{DR}i(E)i(\omega)=0. 
\end{equation}
Thus, on basic forms $\Delta^2=0$, which completes the proof. 
\end{proof}

We conclude that for every Jacobi manifold $(M, \omega, E)$, $(\Omega_B^\bullet(M), d_{DR}, \Delta)$ becomes a mixed complex. This mixed complex is called the \emph{canonical double complex} of the Jacobi manifold~$M$ in~\cite{ChineaMarreroLeon}.

\begin{proposition}
The space of basic differential forms $(\Omega^\bullet_B(M), d_{DR}, \wedge, \Delta)$ forms a dg $\BV$-algebra. 
\end{proposition}

\begin{proof}
This is a direct consequence of the aforementioned arguments, including Lemma~\ref{DeltaSquaredOnJacobi}.
\end{proof}

\begin{theorem}\label{JacobiFrobBasic}
The basic de Rham cohomology of a Jacobi manifold~$(M, \omega, E)$ carries a natural homotopy Frobenius manifold structure, whose rectified dg $\BV$-algebra is homotopy equivalent to the basic de Rham algebra $(\Omega^\bullet_B(M), d_{DR}, \wedge, \Delta)$.
\end{theorem}

\begin{proof}
In view of the previous proposition, the proof is almost identical to that of Theorem~\ref{PoissonFrob}. Indeed, by Formula~\eqref{adjointsquared}, we
have
 $$
[i(\omega),[i(\omega),d_{DR}]]=2i(E)i(\omega),
 $$
so $[i(\omega),[i(\omega),d_{DR}]]=0$ on basic forms. This allows us to duplicate the proof of Theorem~\ref{PoissonFrob}. 
\end{proof}

\begin{corollary}[\cite{ChineaMarreroLeon}]\label{JacobiCollapse}
For every Jacobi manifold~$M$, the spectral sequence for the double complex $(\Omega_{B}^\bullet(M), \allowbreak d_{DR},\Delta)$ degenerates on the first page.
\end{corollary}

\begin{remarks}$ \ $
\begin{itemize}
 \item[$\diamond$] For the so-called regular Jacobi manifolds \cite{ChineaMarreroLeon}, Theorem \ref{JacobiFrobBasic} is literally contained in Theorem \ref{PoissonFrob}. Basically, a Jacobi manifold is \emph{regular} if the space of leaves $\widetilde{M}=M/E$ can be defined as a smooth manifold; in this case, it automatically inherits a Poisson structure from the Jacobi structure on $M$, and $\Omega_B^\bullet(M)\simeq\Omega^\bullet(\widetilde{M})$.
 \item[$\diamond$] In fact, the homotopy Frobenius structure on the de Rham cohomology of a Poisson manifold, as in Theorem \ref{PoissonFrob}, and on the basic de Rham cohomology of a Jacobi manifold, as in Theorem \ref{JacobiFrobBasic}, could also be described in a different way, along the lines of~\cite{KhoroshkinMarkarianShadrin11}.

\noindent
Indeed, in both cases we have a structure of a $\BV/\Delta$-algebra on the de Rham algebra of differential forms (basic differential forms in the case of Jacobi manifolds). In \cite{KhoroshkinMarkarianShadrin11}, an explicit formula for a quasi-isomorphism between the operads $H_\bullet(\overline{\mathcal{M}}_{0,n+1})$ and $\BV/\Delta$ is given. Therefore, the de Rham algebra has a Frobenius manifold structure, and this structure induces a homotopy Frobenius manifold structure on the de Rham cohomology. For details we refer to~\cite{KhoroshkinMarkarianShadrin11}.

\noindent
It is an interesting question whether it is possible to match the two approaches on the level of formulas. Since one of the ways to obtain the aforementioned quasi-isomorphism uses the Givental theory, one natural idea would be to describe the $\BV_\infty$-structure in terms of cohomological field theory and infinitesimal Givental operators. The first step in that direction is made in~\cite{DotsenkoShadrinVallette11}, where this kind of description is given for commutative $\BV_\infty$-algebras.
\end{itemize}
\end{remarks}

In fact, the full de Rham cohomology of a Jacobi manifold carries a homotopy Frobenius manifold structure as well. However, the argument used should be adapted appropriately, since according to Equation~\eqref{deltasquared}, the operator $\Delta^2$ is not equal, but only homotopic to zero. The appropriate notion we shall use here is that of a commutative homotopy $\BV$-algebra.
 
\begin{definition}[Commutative $\BV_\infty$-algebra~\cite{Kravchenko00}]\label{comm-homotopy-bv}
A \emph{commutative $\BV_\infty$-algebra} 
 $$
(A, \allowbreak \wedge, \allowbreak d=\Delta_0, \Delta_1, \Delta_2, \ldots )
 $$ 
is a dg commutative algebra~$A$ equipped with operators $\Delta_n$ of degree $2n-1$ and order at most $n+1$, satisfying 
$$\sum_{i=0}^n \Delta_i \Delta_{n-i}=0, \quad \text{for}\ \ n\ge 0\ . $$
\end{definition}

In particular, the operators $\Delta_n$ of a commutative $\BV_\infty$-algebra $A$ make it a multicomplex.
The following statement is a direct application of a more general homotopy transfer theorem~\cite[Th.~$6.2$]{DrummondColeVallette11}.

\begin{proposition}[{\cite[Prop.~10]{DotsenkoShadrinVallette11}}]\label{CommHomotopyBVFrob}
Let $(A,\wedge,d=\Delta_0, \Delta_1, \Delta_2, \ldots )$ be a commutative $BV_\infty$-algebra admitting a Hodge-to-de Rham degeneration data. The underlying homotopy groups $H(A,d)$ carry a homotopy Frobenius manifold structure extending the induced commutative product.
\end{proposition}

We shall use this result to deduce the following theorem.

\begin{theorem}\label{JacobiFrob}
The de Rham cohomology of a Jacobi manifold carries a natural homotopy Frobenius manifold structure extending the product induced by the wedge product.
\end{theorem}

\begin{proof}
Let us denote $\Delta_0=d_{DR}$, $\Delta_1=\Delta$, and $\Delta_2=i(E)i(\omega)$. Clearly, $\Delta_0^2=0$, and by Formula~\eqref{anticommute}, we have $\Delta_0\Delta_1+\Delta_1\Delta_0=0$. Furthermore, by Formula \eqref{deltasquared}, we have 
$\Delta_1^2+\Delta_2\Delta_0+\Delta_0\Delta_2=0$. Also,
\begin{multline*}
\Delta_1\Delta_2+\Delta_2\Delta_1=\Delta i(E)i(\omega)+i(E)i(\omega)\Delta=\\=
-i(E)\Delta i(\omega)+i(E)i(\omega)\Delta=i(E)[i(\omega),\Delta]=2i(E)^2i(\omega)=0
\end{multline*}
and
 $$
\Delta_2^2=i(\omega)i(E)i(\omega)i(E)=i(\omega)^2i(E)^2=0.
 $$
Therefore, the operators $\Delta_0$, $\Delta_1$, $\Delta_2$ and $\Delta_n=0$ for $n>2$ endow $\Omega^\bullet(M)$ with a structure of a multicomplex. It is clear that $\Delta_0=d_{DR}$ is a differential operator of order at most~$1$, and that  $\Delta_1$ and $\Delta_2$ are differential operators of order at most~$2$ and at most~$3$ respectively. So the de Rham complex of a Jacobi manifold is a commutative homotopy $\BV$-algebra.
By Formula~\eqref{adjointsquared}, 
 $$
[i(\omega),[i(\omega),d_{DR}]]=[i(\omega),\Delta_1]=2\Delta_2  
 $$
and $[i(\omega),[i(\omega),[i(\omega),d_{DR}]]]=[i(\omega),2i(E)i(\omega)]=0$. Therefore,
 $$
e^{i(\omega)z} \, d_{DR}\,  e^{-i(\omega)z}=e^{\mathop{\mathrm{ad}}_{i(\omega)z}}(d_{DR})= \Delta_0 +  \Delta_1  z+\Delta_2 z^2\ .  
 $$
By Theorem~\ref{thm:main}, we conclude that $\Omega^\bullet(M)$ admits a Hodge-to-de Rham degeneration data, so Proposition~\ref{CommHomotopyBVFrob} applies, which completes the proof.
\end{proof}

\bibliographystyle{amsalpha}
\bibliography{bib}

\end{document}